\documentclass{article}
\usepackage[utf8]{inputenc}
\usepackage[margin=1.5in]{geometry}
\usepackage{amsmath}
\usepackage{amsthm, ifthen, amsfonts, amssymb,
srcltx, amsopn, xcolor, enumerate, mathabx, ulem}
\usepackage{mathtools}
\usepackage{enumitem}
\usepackage[colorlinks=true, pdfstartview=FitV, linkcolor=blue, citecolor=blue, urlcolor=blue]{hyperref}
\usepackage[nameinlink]{cleveref}
\usepackage{url}
\usepackage{graphicx}
\usepackage{subcaption}

\newtheorem{thm}{Theorem}[section]

\newtheorem{lem}[thm]{Lemma}
\newtheorem{cor}[thm]{Corollary}
\theoremstyle{definition}
\newtheorem{defn}[thm]{Definition}
\newtheorem{remark}[thm]{Remark}

\newtheorem{claim}[thm]{Claim}

\newtheorem{ques}{Question}
\newcommand{\ov}{\overline}
\newcommand{\cay}{\mathrm{Cay}(G,  S)}
\newcommand{\pres}{\ensuremath{\langle S| R\rangle} }

\newcommand{\presu}{\ensuremath{\langle S\cup S'| R\cup R'\rangle} }

\newcommand*\mc[1]{\mathcal{#1}}

\newcommand{\N}{\mathbb{N}}

\newcommand{\R}{\mathbb{R}}

\DeclarePairedDelimiter\abs{\lvert}{\rvert}%
\newcommand{\diam}{\textup{\textsf{diam}}}

\newcommand{\ol}{\overline}

\newcounter{acomments}

\newcounter{ccomments}

\newcounter{dcomments}

\newcounter{tcomments}

\title{A non-loxodromic Morse element in a Morse local-to-global group}

\author{Carolyn Abbott and Stefanie Zbinden}
\date{}

\begin{document}

\maketitle

\begin{abstract}
    We use small-cancellation techniques to construct a Morse local-to-global group $G$ with an infinite-order Morse element that is not loxodromic in any action of $G$ on a hyperbolic space. In particular, the element cannot be WPD.
\end{abstract}

\section{Introduction}

Introduced by Gromov \cite{Gromov-hyp}, hyperbolicity is the most prominent notion of negative curvature in geometric group theory and has strong algebraic and algorithmic consequences \cite{Gromov-hyp, Paulin, DahGuirardel, Sela, ECHLPT}. Many important groups admit some aspects of negative curvature but are not hyperbolic, including free products of groups, mapping class groups, fundamental groups of many 3--manifolds, certain Artin groups, and the Cremona group. This observation led to the study of various generalizations of hyperbolic groups, such as relatively hyperbolic groups \cite{Farb:rel_hyp, Osin:rel_hyp, Bow:rel-hyp}, acylindrically hyperbolic groups \cite{O:acylindrical, DGO:rotating_families}, and Morse local-to-global (MLTG) groups \cite{russellsprianotran:thelocal}. For any of these generalizations, it is natural to ask which aspects of negative curvature they satisfy.  

In this paper we focus on MLTG groups. Introduced in \cite{russellsprianotran:thelocal}, a main feature of MLTG groups is that they eliminate pathological behavior of Morse geodesics. For example, if a not virtually cyclic MLTG group has a Morse geodesic, it has a Morse element, and if it has a Morse element, it has a subgroup isomorphic to $\mathbb F_2$. This is not true for general groups \cite{fink2017morse, OlshankiiOsinSapir:lacunary}. It is therefore natural to ask whether the elimination of pathological behavior is sufficient to ensure acylindrical hyperbolicity.

\begin{ques}[{\cite[Question~5]{russellsprianotran:thelocal}}]\label{ques:mltg-implies-acy}
    Are non-virtually cyclic MLTG groups with a Morse element acylindrically hyperbolic?
\end{ques}

A standard way to show a group $G$ is acylindrically hyperbolic is to find an action of $G$ on a hyperbolic space and an element $g\in G$ that is a loxodromic \textit{WPD} element in this action \cite[Theorem~1.4]{O:acylindrical}. Defined in \cite{BestFuji:WPD}, the WPD property is a form of properness of the action along the axis of the loxodromic isometry, and WPD elements are necessarily Morse elements of the group \cite{S:hypemb}.  In particular, one could hope that acylindrical hyperbolicity of MLTG groups follows from the fact that Morse elements in MLTG groups are nice enough to be WPD.  Our main result shows that this is not the case.

\begin{thm}\label{thm:mainthm}
    There exists a Morse local-to-global group $G$ and an infinite-order Morse element $a\in G$ such that $a$ is not loxodromic in any action of $G$ on a hyperbolic space. In particular, the element $a$ is not a loxodromic WPD element in any such action.
\end{thm}

The group we construct to prove  Theorem~\ref{thm:mainthm} is an infinitely presented small-cancellation group. Such groups are acylindrically hyperbolic \cite{GH:acylindrical} and have recently been of interest to many researchers \cite{AbbottHume:gen_lox, AbbottHume:hyp_actions, Z:small_cancellation, Gruber:SQ, ArzhDrutu:SC}. Consequently, our result does not provide an answer to Question~\ref{ques:mltg-implies-acy}. However, Theorem~\ref{thm:mainthm} does suggest that Question~\ref{ques:mltg-implies-acy} may have a negative answer, as it eliminates the most intuitive reason to expect a positive one: that Morse elements in MLTG groups are nice enough to be WPD elements.

Our result  shows that any class of groups for which it is true that every Morse element of a group in this class is a loxodromic WPD has to be restrictive. One natural candidate would be MLTG groups which are \emph{Morse detectable}. Roughly, these are MLTG groups for which there exists a Lipschitz map to a hyperbolic space such that a quasi-geodesic in the Cayley graph of the group is Morse if and only if its image is a quasi-geodesic in the hyperbolic space; see \cite{russellsprianotran:thelocal} for the precise definition and details. 

\begin{ques}
    Does there exist a Morse detectable group satisfying Theorem~\ref{thm:mainthm}?
\end{ques}

While this question is open, we suspect that the group we construct in Section~\ref{sec:Construction} is, in fact, Morse detectable. If this is the case, it would be the first example of a group known to have a Morse detectability space but no equivariant Morse detectability space. (No group  satisfying Theorem~\ref{thm:mainthm} can admit an equivariant Morse detectability space). 

It is also not known whether being MLTG is equivalent to having a $\sigma$--compact Morse boundary. He, Spriano and the second author made some progress in this direction by proving that all MLTG groups have $\sigma$-compact Morse boundary and that for $C'(1/9)$--small-cancellation groups the two properties are indeed equivalent \cite{HeSprianoZbinden:sigma}. However, there was some fear that the result for $C'(1/9)$--small-cancellation group is due to the particularly nice behavior of such groups, rather than being part of a broader phenomenon. Theorem~\ref{thm:mainthm} shows that pathological behavior still occurs among these groups, providing hope that the more general result may hold. 

We note that the result of Theorem~\ref{thm:mainthm} is stronger than the element $a$ not being a loxodromic WPD element: the element $a$ is not even \textit{loxodromic} in any action of the group on a hyperbolic space. While there are groups that have no loxodromic isometries in any action on a hyperbolic space, called groups with Property (NL) \cite{BalFFGenSisto}, to the best of our knowledge no (NL) group is known to contain a Morse element. 

\subsection*{Idea of proof}
We are able to find such a counterexample among small cancellation groups because it is well-understood which $C'(1/9)$--small cancellation groups are MLTG: exactly when they do not satisfy the \textit{increasing partial small cancellation} (IPSC) condition \cite{HeSprianoZbinden:sigma}, a property  defined in Definition~\ref{def:IPSC}. Moreover, the IPSC condition is relatively stable under small changes of relators; see, e.g., Lemma~\ref{lem:combination-ipsc}. Furthermore, it is easy to determine which elements are Morse: they are precisely the ones whose intersection function (see Definition~\ref{def:intersection-function}) is sublinear \cite{arzhantseva2019negative}. 

To prove Theorem~\ref{thm:mainthm}, we start with any $C'(1/9)$--small cancellation group $G$ that does not satisfy the IPSC condition, and then modify the relators slightly to ensure the new group is still MLTG and contains an element with the desired property.  We add a generator, say $a$, which will be our Morse element that cannot be loxodromic in any action on hyperbolic space. To prohibit the element $a$ from being loxodromic, we add additional relators: for each possible subword $w$ of a relator of $G$, we add a new relator, which contains a large  power of $a$ and multiple occurrences of the word $w$, arranged in a way that does not violate the small cancellation condition. By choosing the power of $a$ to be sublinear in the length of the relator, we ensure that the element $a$ is Morse. Now consider an action of this new group $G$ on a hyperbolic space $X$. By Lemma~\ref{lem:shortsubpaths}, which states that cycles embedded in a hyperbolic space must have a subsegment of definite length whose endpoints are much closer than its length, each relator of $G$, has to have a subword $w$ whose endpoints are relatively close together in $X$. This then implies that the endpoints of the large power of $a$ in the new relator formed using $w$ also have to be relatively close together in $X$, showing that the orbit of $a$ cannot be quasi-isometrically embedded in $X$. Consequently, $a$ is not a loxodromic isometry of $X$.

\subsection*{Outline of paper}
In Section~\ref{sec:prelim} we provide some background on small-cancellation groups and the increasing partial small cancellation condition. In Section~\ref{sec:cycles-in-hyperbolic} we prove Lemma~\ref{lem:shortsubpaths}, which roughly states that cycles embedded in a hyperbolic space must have a subsegment of definite length whose endpoints are much closer than its length. In Section~\ref{sec:Construction}, we construct the group $G$ from Theorem~\ref{thm:mainthm}.

\subsection*{Acknowledgements}
  We would like to thank Cornelia Dru\c{t}u, Alessandro Sisto and Davide Spriano for interesting conversations. 
 The first author was partially supported by NSF grants DMS-2106906 and DMS-2340341.

\section{Preliminaries}\label{sec:prelim}
Let $X$  be a geodesic metric space. For points $x, y\in X$, we denote by $[x, y]$ a fixed choice of geodesic from $x$ to $y$. By an abuse of notation, we denote the image $\mathrm{Im}(p)$ of a path $p\colon I \to X$   by $p$ and its endpoints by $p^-$ and $p^+$. Given $x, y\in p$, we denote by $[x, y]_p$ a choice of subsegment $p[s, t]$ such that $p(s) = x$ and $p(t) = y$ and $s\leq t$. We denote the length of (the domain of) $p$ by $\abs{p}$. Given a cycle $K$, we say a 1--Lipschitz map $C\colon K \to X$ is an \textit{embedde cycle}.

\subsection{The Morse local-to-global property}\label{sec:mltg}

In this section, we define the Morse local-to-global property, a property introduced in \cite{russellsprianotran:thelocal}.

    A function $M\colon \R_{\geq 1}\to \R_{\geq 0}$ is called a \emph{Morse gauge} if it is non-decreasing and continuous.
    A quasi-geodesic $\gamma$ is called \emph{$M$-Morse} for some Morse gauge $M$ if every $Q$--quasi-geodesic $\lambda$ with endpoints $\gamma(s)$ and $\gamma(t)$ stays in the closed $M(Q)$--neighbourhood of $\gamma[s, t]$. A quasi-geodesic is called \emph{Morse} if it is $M$--Morse for some Morse gauge $M$.  An element $a$ of a finitely generated group $G$ is \textit{Morse} if the subset $\{a^i\}_{i\in \mathbb Z}$ in some (equivalently, any) Cayley graph of $G$ is a Morse quasi-geodesic.

We say that a path $p \colon I\to X$ \emph{$L$--locally satisfies a property $(P)$} if for every $s, t \in I$ with $\abs{t - s} \leq L$ the subpath $p[s, t]$ has the property $(P)$. The following property, introduced by Russell, Spriano, and Tran in \cite{russellsprianotran:thelocal}, generalizes the property of Gromov hyperbolic spaces that every local quasi-geodesic is a global quasi-geodesic.  

\begin{defn}[Morse local-to-global]\label{def_morse-local-to-global}
    A metric space $X$ satisfies the \emph{Morse local-to-global} (MLTG) property if the following holds. For any constant $Q\geq 1$ and Morse gauge $M$, there exists a scale $L$, a constant $Q'\geq 1$ and a Morse gauge $M'$ such that every path that is $L$--locally an $M$--Morse $Q$--quasi-geodesic is an $M'$--Morse $Q'$--quasi-geodesic. 
\end{defn}

\subsection{Small cancellation}

We now define the notions from small-cancellation necessary for this  paper. For further background on small-cancellation, we refer the reader to \cite{lyndon1977combinatorial}. 

\textbf{Notation and Conventions.} For the rest of this paper, unless specified otherwise, $ S$ denotes a finite set of formal variables, ${ S}^{-1}$ its formal inverses and $\ov{ S}$ the symmetrised set $  S \cup { S}^{-1}$. A word $w$ over $ S$ (respectively $\ov{ S}$) is a finite sequence of elements in $ S$ (respectively $\ov{ S}$). By an abuse of notation, we sometimes allow words to be infinite.

A word $w$ over $\ov { S}$ is \emph{cyclically reduced} if it is reduced and all its cyclic shifts are reduced. Given a set $ R$ of cyclically reduced words, we denote by $\overline{ R}$ the cyclic closure of $ R\cup  R^{-1}$. If $ R = \{w\}$ we sometimes denote $\ov{{R}}$ by $\ov w$. Given a cyclically reduced word $r$, we say that $w$ is a \emph{cyclical subword} of $r$ if $w$ is a subword of a word in $\ov{r}$.

\begin{defn}[Piece]
    Let $ S$ be a finite set and let $ R$ be a set of cyclically reduced words over $\ov{ S}$. A word $p$ is a \emph{piece} if there exists distinct words $r, r'\in \overline{ R}$ such that $p$ is a prefix of both $r$ and $r'$. We say that $p$ is a \emph{piece of a word} $r\in \ov{ R}$ if $p$ is a piece and a subword of $r$.
\end{defn}

\begin{defn}[$C'(\lambda)$ condition]
    Let $\lambda >0$ be a constant. A set $ R$ of cyclically reduced words satisfies the \emph{$C'(\lambda)$--small-cancellation condition} if for every word $r\in \overline{ R}$ and every piece $p$ of $r$ we have $\abs{p}<\lambda\abs{r}$. If $ R$ satisfies the $C'(\lambda)$--small-cancellation condition we call the finitely generated group $G = \pres$ a \textit{$C'(\lambda)$--group}.
\end{defn}

\subsubsection{Morse geodesics in small-cancellation groups} In this section we highlight some consequences of \cite{arzhantseva2019negative}, which allow us to characterize Morse geodesics in small-cancellation groups in terms of their intersections with relators.  Given a finitely generated group $G = \pres$, every edge path in $\cay$ is labelled by a word in $\ov{S}$.

\begin{defn}[Intersection function]\label{def:intersection-function}
Let $G = \pres$ be a $C'(1/6)$--group. Let $\gamma$ be an edge path in $\cay$. The \emph{intersection function} of $\gamma$ is the function $\rho \colon \N\to \R_+$ defined by 
\begin{align*}
    \rho(t) = \max_{\substack{r\in\ov{ R} \\ \abs{r}\leq t}}\left\{\abs{w}\,\Big\vert\,  \text{$w$ is a subword of $r$ and $\gamma$}\right\}.
\end{align*}
\end{defn}

The following lemma shows the relationship between Morseness and the intersection function;  it is a consequence of \cite[Theorem 1.4]{ACHG:contraction_morse_divergence} and \cite[Corollary 4.14, Theorem 4.1]{arzhantseva2019negative}.

\begin{lem}[{\cite[Theorem 4.1 \& Corollary 4.14]{arzhantseva2019negative}}]\label{lem:4.14} Let $G = \pres$ be a $C'(1/6)$--group. A geodesic in $\cay$ is Morse if and only if its intersection function is sublinear.
\end{lem}

Further, as shown for example in \cite{Z:small_cancellation}, having an intersection function which is relatively small implies being a geodesic. 

\begin{lem}[{\cite[Lemma~2.7]{Z:small_cancellation}}]\label{lemma:geodesic_condition}
Let $G = \pres$ be a $C'(1/6)$--group, and let $\gamma$ be a path in $\cay$ labelled by a reduced word $w$. If the intersection function $\rho$ of $\gamma$ satisfies 
\begin{align}\label{eq:geodesic_condition}
    \rho(t)\leq t/3 \qquad\text{for all $t\geq 0$},
\end{align}
then $\gamma$ is a geodesic.
\end{lem}

\subsubsection{Increasing partial small-cancellation condition}

In \cite{Z:small_cancellation}, the notion of a $C'(1/f)$--group for certain functions $f$ is introduced in order to define the increasing partial small-cancellation condition (see Definition~\ref{def:IPSC}). A non-decreasing function $f\colon \N \to \R_+$ is viable if $f(n)\geq 6$ for all $n$ and $\lim_{n\to \infty} f(n) = \infty$.

\begin{defn}\label{def:c1f} Let $f$ be a viable function. A finitely generated group $G = \pres$ is a \textit{$C'(1/f)$--group} if $ R$ is infinite and for every piece $p$ of a relator $r\in \ov{ R}$ we have that $\abs{p}<\abs{r} / f(\abs{r})$. 
\end{defn}

In other words, in $C'(1/f)$--groups, pieces of large relators are smaller fractions of their relators than pieces of smaller relators. The condition that $f(n)\geq 6$ ensures that all $C'(1/f)$--groups are $C'(1/6)$--groups. Instead of requiring that a set of relators $ R$ satisfies the $C'(1/f)$ condition as a whole, we can also require that certain words satisfy the $C'(1/f)$ condition with respect to $ R$.

\begin{defn}\label{def:loacl:c1f}
    Let $f$ be a viable function, let $x$ be a reduced word and let $ R$ be a set of cyclically reduced words. The pair $(x,  R)$ satisfies the \textit{$C'(1/f)$--condition} if every piece $p$ of a relator $r\in \ov { R}$ which is a subword of $x$ satisfies $\abs{p}< \abs{r}/f(\abs{r})$.
\end{defn}

\begin{remark}\label{rem:cf-inclusion}
    If $(x,  R)$ satisfies the $C'(1/f)$--condition, and $ R'$ is a subset of $ R$, then $(x,  R')$ also satisfies the $C'(1/f)$--condition. Further, if $y$ is a subword of $x$ and $(x,  R)$ satisfies the $C'(1/f)$--condition, then so does $(y,  R)$.
\end{remark}

The following definition introduced in \cite{Z:small_cancellation} quantifies having sufficiently long subwords $w$ of longer and longer relators such that $(w,  R)$ satisfies the $C'(1/f)$--small-cancellation condition. 

\begin{defn}[IPSC]\label{def:IPSC}
     A $C'(1/6)$--group $G = \pres$ satisfies the \emph{increasing partial small-cancellation condition (IPSC)} if for every sequence $(n_i)_{i\in \N}$ of positive integers, there exists a viable function $f$ such that the following holds. For all $K\geq 0$ there exists $i\geq K$ and a relator $r = xy\in \ov{ R}$ satisfying: 
    \begin{enumerate}[label= \roman*)]
        \item $\abs{r}\geq n_i$,
        \item $\abs{x}\geq \abs{r}/i$,
        \item the pair $(x,  R)$ satisfies the $C'(1/f)$--small-cancellation condition.
    \end{enumerate}
\end{defn}

The motivation for the IPSC condition is its connection to the Morse local-to-global property and the Morse boundary of the group $G$.

\begin{lem}[{\cite[Theorem~C]{HeSprianoZbinden:sigma}} and {\cite[Theorem~B]{Z:small_cancellation}}]\label{lem:equivalence:ipsc-sigma-compact-mltg} For  a $C'(1/9)$--group $G = \pres$,  the following are equivalent:
\begin{enumerate}
    \item  $G$ satisfies the Morse local-to-global property,
    \item $\pres$ does not satisfy the IPSC condition, and
    \item the Morse boundary $\partial_* G$ is $\sigma$-compact.
\end{enumerate}
\end{lem}

The following lemma quantifies the fact that the IPSC condition is invariant under small perturbations of the relators of a group. This allows for flexibility when constructing groups with specific properties that do not satisfy the IPSC condition. This lemma is our main tool to show that the groups we construct do not satisfy the IPSC condition and hence are Morse local-to-global groups.

\begin{lem}\label{lem:combination-ipsc}
    Let $G = \pres$ be a group that does not satisfy the IPSC condition. Let $ S'$ be a set of formal variables disjoint from $ S$, let $N, B$ be integers and let $\rho$ be a sublinear function. Further, let $ R'$ be a set of relators over $\ol{S\cup S'}$ such that for each relator $r'\in  R'$ we have 
    \begin{align}
        r' = \prod_{i=1}^ku_iv_i,
    \end{align}
    for some $1\leq k \leq N$ and such that for each $1\leq i \leq k$,
    \begin{enumerate}[label= (\alph*)]
        \item \label{prop:a}$u_i$ is a prefix of a relator $r_i\in \ov{ R}$,
        \item \label{prop:b}$\abs{u_i}\geq \abs{r_i}/B$,
        \item $B\abs{r'}\geq \abs{r_i}\geq \abs{r'}/B$, and \label{prop:similar-size}
        \item $v_i$ is a (possibly empty) word over $\ov{ S \cup S'}$ with $\abs{v_i}\leq \rho(\abs{r'})$.\label{prop:vi_is_subl}
    \end{enumerate}
    Then $\presu$ does not satisfy the IPSC condition. 
\end{lem}

It is possible that $\presu$ does not satisfy the $C'(1/6)$--condition. In that case the statement is vacuously true, since  the IPSC condition implies  the $C'(1/6)$--condition.

\begin{proof}[Proof of Lemma~\ref{lem:combination-ipsc}]
    Assume toward a contradiction that $\presu$ satisfies the IPSC condition. We will show that this implies that $\pres$ satisfies the IPSC condition. 

    Let $(n_i)_{i\in \N}$ be a sequence of integers. Define a sequence $(n'_i)_{i\in \N}$ such that 
    \begin{enumerate}[label=(\roman*)]
        \item if $t\geq n_i'$, then $\rho(t) < t/(i(2N+1))$,\label{prop:subcons}
        \item $n_i'\geq Bn_j$ for all $1\leq j\leq i(2N+1)B$.\label{prop:defni2}
    \end{enumerate}
    Note that we can find a sequence $n_i'$ that satisfies \ref{prop:subcons} (and hence which satisfies both \ref{prop:subcons} and \ref{prop:defni2}) because $\rho$ is sublinear. 
    Since $\presu$ satisfies the IPSC condition, there exists a viable function $f$ such that for each integer $K$, there exists $i\geq K$ and a relator $r_i' = x_i'y_i'\in \ov{ R \cup  R'}$ satisfying 
    \begin{enumerate}[label=(\arabic*)]
        \item $\abs{r_i'}\geq n_i'$,\label{p2}
        \item $\abs{x_i'}\geq \abs{r_i'}/i$,\label{p1}
        \item $(x_i',  R \cup  R')$ satisfies the $C'(1/f)$--condition.\label{p3}
    \end{enumerate}
    If $r_i'\in \ov{ R}$, then define $r_i=r_i'$, $x_i = x_i'$ and $l = i$. If, on the other hand, $r_i' \in  R'$, then define $l= i(2N+1)B$. In this case, there exists an integer $k$ and words $u_j, v_j, w_j$ for $1\leq j \leq k$ as in the statement such that $r_i'\in \ov{u_1v_1\ldots u_kv_k}$. By the pigeon hole principle, there exists $j$ such that $u_j$ or $v_j$ share a common subword, which we call $x_i$, with $x_i'$ that is of length at least $\abs{x_i'}/(2N+1)$. Consequently, $\abs{x_i}\geq \abs{r_i'}/(i(2N+1))$.  On the other hand, by \ref{prop:subcons} in our choice of $n_i'$ and \ref{prop:vi_is_subl}, we see that $\abs{v_j} \leq \rho(\abs{r_i'}) <\abs{r_i'}/i(2N+1)$.  Hence $x_i$ cannot be a subword of $v_j$ for any $j$. Thus $x_i$ is a subword of $u_j$ for some $j$. In particular, we can choose $r_i\in \ov{w_j}$ such that $x_i$ is a prefix of $r_i$. We have that 
    \[
    \abs{x_i}\geq \frac{\abs{x_i'}}{(2N+1)}\geq \frac{\abs{r_i'}}{i(2N+1)}\geq \frac{\abs{r_i}}{{i(2N+1)B}} = \frac{\abs{r_i}}{l},
    \]
    where the first inequality holds by the definition of $x_i$, the second  due to \ref{p1} and the third due to \ref{prop:similar-size}. Further, we have that 
    \[
    \abs{r_i}\geq \frac{\abs{r_i'}}{B}\geq \frac{n_i'}{B}\geq n_{i(2N+1)B} = n_l,
    \]
    where the first inequality holds due to \ref{prop:similar-size}, the second due to \ref{p2} and the third due to \ref{prop:defni2}. 

    Regardless of whether $r_i'\in \ov{ R}$ or $r_i'\in \ov{ R'}$ we have defined a relator $r_i\in \ov{ R}$, a prefix $x_i$ of $r_i$ and an integer $l\geq i\geq K$ such that 
    \begin{enumerate}[label = \Roman*)]
        \item $\abs{r_i}\geq n_l$,
        \item $\abs{x_i}\geq \abs{r_i}/l$, and
        \item $(x_i,  R)$ satisfies the $C'(1/f)$--condition.
    \end{enumerate}
    Indeed, the first two properties follow directly from \ref{p2} and \ref{p1} if $r_i'\in \ov{ R}$ and are shown to hold above if $r_i'\in\ov{ R'}$. The third property follows from \ref{p3} and  Remark~\ref{rem:cf-inclusion}.

    Therefore, $\pres$ satisfies the IPSC condition, which concludes the proof.     
\end{proof}

\section{Embedded cycles in hyperbolic spaces}\label{sec:cycles-in-hyperbolic}

In this section, we show that if a cycle $C$ is embedded in a hyperbolic space, then there is a subpath of $C$ with definite length whose endpoints are much closer than its length; Lemma~\ref{lem:shortsubpaths}.  We begin with some preliminary results. 

\begin{lem}[{\cite[Proposition III.H.1.6]{BH}}]\label{lem:hyperbolic-paths}
    Let $X$ be a $\delta$-hyperbolic space and let $\lambda : I \to X$ be a $1$-Lipschitz path from $x$ to $y$. Then $[x, y]$ is contained in the $f(\abs{\lambda})$ neighbourhood of $\lambda$, where $f(n) = \delta\abs{\log_2(n)}+1$.
\end{lem}

\begin{cor}\label{cor:hyperbolic-paths}
    Let $X$ be a $\delta$-hyperbolic space, let $x, y\in X$ and let $D\geq 0$ be a constant. Let $\lambda \colon  I \to X$ be a $1$-Lipschitz path from $x$ to $z$ for some $z$ with $d([x, y], z)\leq D$ and let $u$ be a point on $[x, y]$ with $d(z, u)\leq D$. Then $[x, u]_{[x, y]}$ is contained in the $f(\abs{\lambda} + D)+D$ neighbourhood of $\lambda$, where $f$ is as in Lemma~\ref{lem:hyperbolic-paths}.
\end{cor}
\begin{proof}
    Let $\lambda'$ be the concatenation of $\lambda$ and $[z, u]$, so that $\abs{\lambda'} \leq \abs{\lambda} + D$. Applying Lemma~\ref{lem:hyperbolic-paths} yields that $[x, u]_{[x, y]}$ is in the $f(\abs{\lambda'})\leq f(\abs{\lambda} + D)$ neighbourhood of $\lambda'$. Since $\lambda'$ is in the $D$--neighbourhood of $\lambda$, the statement follows.
\end{proof}

\begin{remark}\label{rem:hyperbolic-paths}
    Consider the function $f' (n) = f(n + f(n)) + f(n)$, where $f$ is the function from Lemma~\ref{lem:hyperbolic-paths}. Since $f$ is logarithmic, $f(n)\leq n$ for large $n$, and hence $f'$ is logarithmic as well. In particular, applying Corollary~\ref{cor:hyperbolic-paths} to $\lambda$ and $D$ for $D\leq f(\abs{\lambda})$ yields that $[x, u]_{[x, y]}$ is contained in the $f'(\abs{\lambda})$--neighbourhood of $\lambda$.
\end{remark}

\begin{lem}\label{lem:shortsubpaths}
    Let $g$ be a sublinear function that is superlogarithmic. For all integers $U\geq 1$ there exists an integer $L\geq U$ such that for all $\delta\geq 0$ there exists $N = N(\delta, U, L)$ such that the following holds. If $C$ is an embedded cycle in a $\delta$-hyperbolic space and $\abs{C}\geq N$, then there exists a subsegment $\lambda$ of $C$ with endpoints $\lambda^-$ and $\lambda^+$ such that   
    \begin{align}\label{eq1}
        \frac{\abs{C}}{L}\leq \abs{\lambda}\leq \frac{\abs{C}}{U}
    \end{align}
    and 
    \begin{align}\label{eq2}
          d(\lambda^-, \lambda^+)\leq g(\abs{C}).
    \end{align}
\end{lem}

\begin{proof}
     We will show that the statement holds for $L = 32U$. Let $f$ and $f'$ be the functions from Lemma~\ref{lem:hyperbolic-paths} and Remark~\ref{rem:hyperbolic-paths}, respectively. Let $\delta\geq 0$, and let $N\geq 2$ be such that for all $n\geq N$
    \begin{align}\label{eq:g-bound}
        g(n)\geq 12f(n) + 3\delta + f'(n) + 3.
    \end{align}
        Such an integer $N$ exists because $g$ is superlogarithmic while $f$ and $f'$ are logarithmic. 

    Let $X$ be a $\delta$--hyperbolic space, and let $C$ be an embedded cycle in $X$ with $\abs{C}\geq N$. Define $D = f(\abs{C})$ and $D' = f'(\abs{C})$. Note that since $N\geq 2$, we have that $D\geq \delta$. By Lemma~\ref{lem:hyperbolic-paths}, for any 1--Lipschitz path $\gamma$ with $\abs{\gamma}\leq \abs{C}$ we know that $[\gamma^-, \gamma^+]\subset \mc N_{D}(\gamma)$, and by Remark~\ref{rem:hyperbolic-paths}, for any pair of points $x \in \gamma$ and  $y\in [\gamma^-, \gamma^+]$ with $d(x, y)\leq D$, we have that $[y, \gamma^+]_{[\gamma^-, \gamma^+]}\subset \mc N_{D'}([x, \gamma^+]_{\gamma})$.

    Let $M_1 = 4U$ and $M_2 = 16U$. We will prove the following claim.

    \begin{claim}\label{claim:recursion}
        Let $\gamma\subset C$ be a subpath such that 
        \begin{enumerate}[label = \arabic*)]
            \item $d(\gamma^+, \gamma^-)\leq 2D$\label{close-end}
            \item $\abs{\gamma}\geq \abs{C}/M_1$\label{long}
        \end{enumerate}
        Then at least one of the following holds. 
        \begin{enumerate}[label=\roman*)]
            \item There exists a subpath $\gamma' \subset C$ which satisfies \eqref{eq1} and \eqref{eq2}. \label{base-case}
            \item There exists a subsegment $\gamma'\subset \gamma$ such that \ref{close-end} and \ref{long} hold for $\gamma'$ and $\abs{\gamma'} \leq \abs{\gamma}-1$. \label{rec-path}
        \end{enumerate}
    \end{claim}

    \textit{Proof of Claim.}
        If $\diam(\gamma)\leq g(\abs{C})$ we are done, since any subsegment $\lambda$ of $\gamma$ of length $\abs{\lambda} = \abs{C}/M_2$ satisfies \eqref{eq1} and \eqref{eq2}. Hence we can assume that there exists a point $y$ on $\gamma$ which satisfies $d(\gamma^+, y)\geq g(\abs{C})/2$. It follows from \eqref{eq:g-bound}  that $d(\gamma^+, \gamma^-)\leq g(\abs{C})/6$ and hence $d(\gamma^-, y)\geq g(\abs{C})/3$. 

        We denote $\gamma^-$ and $\gamma^+$ by $x_1$ and $x_2$ and the paths $[x_1, y]_{\gamma}$ and $[x_2, y]_{\gamma}$ by $\gamma_1$ and $\gamma_2$, respectively. Note that by \eqref{eq:g-bound}, $\abs{\gamma_i}\geq 1$ for $i=1, 2$.

        Let $i=1, 2$, and let $z\in [x_i, y]$. We say that a connected component of $\gamma_i - B_D(z)$ whose endpoints are in the $D$--neighbourhood of $z$ is a \emph{$z$--excursion of $\gamma_i$}.   

    \begin{figure}
        \centering
        \begin{subfigure}{.45\linewidth}
        \centering
        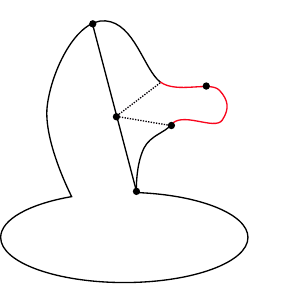
        \caption{ $a_1$ lies on a $u_1$--excursion $\lambda$ (red).}
        \end{subfigure} \hfill
        \begin{subfigure}{.45\linewidth}
        \centering
        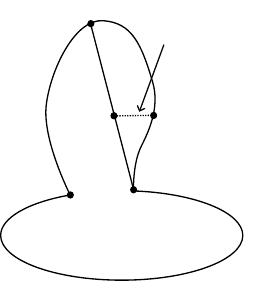
        \caption{$a_1$ is in the $D$--neighborhood of $[x_1,y]$.}
        \end{subfigure}
        \caption{Possible configurations of $a_1,a_1'$, and $u_1$ in the proof of Lemma~\ref{lem:shortsubpaths}.  Finding $a_2,a_2',u_2$ is similar.}
        \label{fig:Case1}
    \end{figure}

       Suppose that for $i=1$ or $2$ and some $z\in [x_i, y]$ there exists a $z$--excursion $\lambda$ of $\gamma_i$ which satisfies $\abs{\lambda}\geq \abs{C}/M_1$.  Then, by definition of a $z$--excursion, we have $d(\lambda^+, z)\leq D$ and $d(\lambda^-, z)\leq D$, and hence $d(\lambda^-, \lambda^+)\leq 2D$. Since $\lambda$ is a subset of $\gamma_i$, we have that $\abs{\lambda} \leq \abs{\gamma} - \abs{\gamma_j} \leq \abs{\gamma}-1$ for $j\neq i$. Hence $\lambda$ satisfies the requirements of  \ref{rec-path}, which concludes the proof of the claim.

        Thus we assume for the remainder of the proof of the claim that for $i\in \{1,2\}$,  for all $z\in [x_i, y]$ and for all $z$-excursions $\lambda$, we have $\abs{\lambda}\leq \abs{C}/M_1$. Assume without loss of generality that $\abs{\gamma_1}\geq \abs{\gamma_2}$. In particular, $\abs{\gamma_1}\geq \abs{\gamma}/2\geq \abs{C}/M_2$.

        For $i\in \{1,2\}$, choose $a_i\in \gamma_i$ such that $[a_i, y]_{\gamma_i}$ has length $\abs{C}/M_2$. While such a point $a_1$ always exists, the point $a_2$ is not well-defined if $\abs{\gamma_2}< \abs{C}/M_2$. In this case, define $a_2 = x_2$. If $d(a_i, [x_i, y])\leq D$, define $a_i' = a_i$, and define $u_i$ to be a point on $[x_i, y]$ with $d(u_i, a_i')\leq D$. If $d(a_i, [x_i, y]) > D$, then $a_i$ lies on a $u_i$--excursion $\lambda$ of $\gamma_i$ for some $u_i\in [x_i, y]$. In this case define $a_i'$ as the endpoint of $\lambda$ which lies between $a_i$ and $x_i$ on $\gamma_i$. If $\abs{\gamma_2} < \abs{C}/M_2$ we define $a_2' = u_2 = x_2$ instead. This is depicted in Figure~\ref{fig:Case1}.

        Suppose that $d(y, u_1) \leq d(y, u_2)+2D$. Let $u_1'$ be the point on $[x_2, y]$ with $d(y, u_2) =\max\{0, d(y, u_1) - 2D-\delta \}$. In particular, $d(y, u_1')\leq d(y, u_2)$ and by the hyperbolicity of $X$, $d(u_1, u_1')\leq 2D+3\delta$.\footnote{If $d(u_1, y)\leq 2D + \delta$, then this is immediate. So we assume $d(u_1, y) > 2D + \delta$, which in particular implies that $d(u_1', y) = d(u_1, y) - 2D - \delta$ and hence $d(u_1', [x_1, x_2]) > \delta$. Thus, by hyperbolicity, there exists a point $u_1''$ on $[x_1, y]$ with $d(u_1', u_1'')\leq \delta$. By the triangle inequality, $d(u_1', y)-\delta\leq d(u_1'', y)\leq d(u_1', y) + \delta$ and in particular, $d(u_1, u_1'')\leq 2D + 2\delta$. Implying that $d(u_1, u_1')\leq 2D + 3\delta$.}
        By the choice of $D$ and $D'$, there exists a point $a_2''\in [a_2', y]_{\gamma_2}$ with $d(a_2'', u_1')\leq D'$. Let $\lambda'$ be the subsegment of $\gamma$ from $a_1'$ to $a_2''$.  See Figure~\ref{fig:Case2}. By construction, $a_2$ lies between $y$ and $a_2'$ on $\gamma$. Hence

    \begin{figure}
        \centering
        \def\svgwidth{2in}
        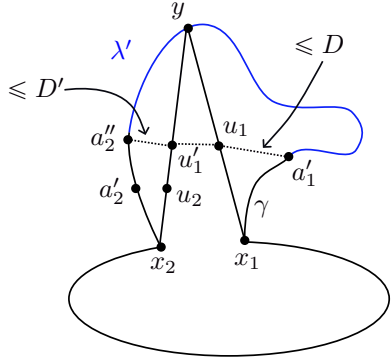
        \caption{Finding the desired subpath $\lambda'$ (blue) in the proof of Lemma~\ref{lem:shortsubpaths} when $d(y,u_1)<d(y,u_2)+2D$.}
        \label{fig:Case2}
    \end{figure}
    
        \begin{align*}
             \abs{[y, a_2']_{\gamma}}\leq  \abs{[y, a_2]_{\gamma}} +  \abs{[a_2, a_2']_{\gamma}}\leq \abs{C}/M_2 + \abs{C}/M_1,
        \end{align*}
        where for the second inequality we used the definition of $a_2$ and the assumption that any $\lambda$-excursion has length at most $\abs{C}/M_1$. Similarly, $\abs{[a_1', y]_{\gamma}}\leq \abs{C}/M_1 + \abs{C}/M_2$. By definition, $\lambda'$ is a subpath of $[a_1', a_2']_{\gamma}$, yielding the following upper bound on $\abs{\lambda'}$:
        \begin{align}\label{eq:upper-bound}
            \abs{\lambda'}\leq 2\left(\abs{C}/M_2 + \abs{C}/M_1)\right)\leq \abs{C}/U.
        \end{align}
    
        On the other hand $[a_1, y]_{\gamma}$ is a subsegment of $\lambda'$. Thus we have the following lower bound on $\abs{\lambda'}$:
        \begin{align} \label{eq:lower-bound}
            \abs{\lambda'} \geq \abs{C}/M_2\geq \abs{C}/L.
        \end{align} 
        Furthermore, using \eqref{eq:g-bound}, we see that
    \begin{align}\label{eq-distab}
        d(\lambda'^+, \lambda'^-) \leq d(a_2'', u_1')+d(u_1', u_1) + d(u_1, a_1')\leq D'+2D+3\delta+D\leq g(\abs{C}).
    \end{align}
    Thus, we have found a subsegment $\lambda'$ of $C$ satisfying \eqref{eq1} and \eqref{eq2}, as desired. 
    
    If, on the other hand, $d(y, u_1) \geq d(y, u_2)+2D$, then it must be the case that $u_2\neq x_2$ and hence $\abs{\gamma_2} \geq \abs{C}/{M_2}$. In that case an analogous argument shows that there exists a subsegment of $C$ satisfying \eqref{eq:upper-bound}, \eqref{eq:lower-bound} and \eqref{eq-distab} and hence satisfying \eqref{eq1} and \eqref{eq2}. This completes the proof of the claim.
    \hfill$\blacksquare$\\

    To complete the proof of the lemma, we apply   Claim~\ref{claim:recursion} to $\gamma=C$, which satisfies \ref{close-end} and \ref{long}.  If \ref{base-case} holds, then we have found our desired subpath $\lambda=\gamma'$.  If \ref{rec-path} holds, then we inductively apply Claim~\ref{claim:recursion} to the subpath $\gamma'$.  Since the length of $\gamma'$ is decreasing by at least 1 at each step, this process will terminate in finitely many steps, at which point we will have produced the desired subsegment $\lambda$ of $C$.
\end{proof}

\section{Construction of the groups}\label{sec:Construction}

Let $N, M, L, U, V\geq 36$ be integers such that $L\geq U$. Let $f$ be a function that is sublinear but superlogarithmic. Fix a $C'(4/N)$--group $G = \pres$ that does not satisfy the IPSC condition. 

\begin{remark}
    Such groups exist. For example, the groups constructed in \cite[Section~4]{Z:small_cancellation} do not satisfy the IPSC condition. They are $C'(1/7)$--groups, but it follows from the proofs in that paper that they are, in fact, $C'(4/N)$--groups.
\end{remark}

We now construct a new group $G'=\langle S' \mid R'\rangle$ from $G$ by carefully adding generators and relators. In Section~\ref{sec:SC&MLTG}, we show that $G'$ is a $C'(1/9)$--group\footnote{In fact, we show that $G'$ is a $C'(1/\lambda)$--group for some arbitrarily large $\lambda$, but we will not need this fact in this paper.} and a Morse local-to-global group. In Section~\ref{sec:main-theorem-proof}, we  show that $G'$ contains a Morse element that is not loxodromic in any action of $G'$ on a hyperbolic space, thereby proving Theorem~\ref{thm:mainthm}.

Let $T$ be a set of two formal variables that are distinct from those in $S$, and let $a$ be a formal variable which is distinct from any formal variables in $S\cup T$. Define $S' = S\cup T \cup \{a\}$. Let $ \mc T$ be the set of words over $T$, and write $\mc T = \{u_1, u_2, \ldots\}$ such that $\abs{u_i}\leq \abs{u_j}$ for all $i \leq j$.

Let $R_1\subset R$ be a set of relators of length at least $V$ and which contains at most one relator of each length. For each  $r\in R_1$, we denote by $W_r$ the set of cyclic subwords $w$ of $r$ whose length $\abs{w}$ satisfies 
\[
\abs{r}/L < \abs{w} < \abs{r}/U.
\]
Since there are at most $\abs{r}^2$ cyclic subwords of $r$, we have that $\abs{W_r}\leq \abs{r}^2$.

Iterate (increasing in length) over all relators $r\in R_1$. For each of those relators $r$, iterate over all words $w\in W_r$ and define a word  
\[
c_w := a^{f(\abs{r})}\Pi_{i=1}^M w u_{i_w +i},
\]
where $u_{i_w+1}, \ldots, u_{i_w +M}$ are the lowest indexed words in $\mc T$ which have not yet been used. Define 
\[
C_r = \{ c_w \mid  w\in W_r\}.
\]
Finally, define $C = \bigcup_{r\in R_1} C_r$ and $R'=R\cup C$. Let
\[
G' = \langle S' \mid R'\rangle.
\]

\subsection{Small cancellation and Morse local-to-global properties}
\label{sec:SC&MLTG}

 Our first goal is to show that $G'$ is a small cancellation group. To do this, we first show that the subwords $u_j$ in each relator $c_w$ have negligible length. Recall that by definition, each $r\in R_1$ has length at least $V$.

\begin{lem}\label{lem:length-bound-ui}
If $r\in R_1$ is a relator and $w\in W_r$, then $\abs{u_{i_w+i}}\leq \log_2(2M\abs{r}^3)+1$ for all $1\leq i \leq M$. Moreover, for $V$ large enough, we have $\abs{u_{i_w + i}}\leq \abs{c_w}/(2M)$ for all $1\leq i\leq M$.
\end{lem}
\begin{proof}
    Let $n = \abs{r}$. Recall that $R_1$ contains at most one relator for each length, and for each relator $r'\in R$ whose length is at most $n$ we have that $\abs{W_{r'}}\leq \abs{r'}^2\leq n^2$.  It follows that $i_w\leq Mn^3$ and hence $i_w+M\leq 2Mn^3$. Since the words in $\mc T$ are ordered by their length, we have that $\abs{u_{i}}\leq \log_2(i)+1$. Hence $\abs{u_{i_w +i}}\leq \log_2(2Mn^3)+1$, which concludes the first part of the statement. 
    
    For the second part, note that $\abs{c_w}\geq M\abs{r}/L = Mn/L$. Since these bounds do not depend on $V$ and $n/(2L)$ (which is a lower bound on $\abs{c_w}/(2M)$) grows faster than $\log_2(2Mn^3)+1$, we have $\abs{u_{i_w+i}}\leq \abs{c_w}/(2M)$ for large enough $V$. 
\end{proof}

\begin{lem}\label{lem:small-cancellation-condition-satisfied}
    Let $\lambda = \max\{N, M, U\}/4$. If $V$ is large enough, then the group $G' = \langle S' \mid R'\rangle$ is a $C'(1/\lambda)$--group. In particular, $G'$ is a $C'(1/9)$--group.
\end{lem}
\begin{proof}
Let $p$ be a subsegment of two relators $r_1$ and $r_2$. If $r_1, r_2\in R$, then $\abs{p} <\frac{4\abs{r_1}}{N}$ since $G = \pres$ is a $C'(4/N)$--group. By symmetry, it remains to show that $\abs{p}< 4\abs{r_1}/(\max\{N, M, U\})$ and $\abs{p}< 4\abs{r_2}/(\max\{N, M, U\})$ in the case where $r_1\in C$. 

If $r_1 \in C$, then
\[
r_1 = a^k\Pi_{i=1}^M w u_{i_w +i}
\]
for some relator $r\in R_1$, cyclic subword $w\in W_r$ and $k = f(\abs{r})$. 

\textbf{Case 1:} $r_2\in R$. In this case, $p$ has to be a subword of $w$ or its inverse. In particular, $\abs{p}\leq \abs{w} < \abs{r_1}/M$. If $\ov{r} = \ov{r_2}$, then we know that $\abs{w}\leq \abs{r}/U < 4\abs{r_2}/U$. If $\ov{r}\neq \ov{r_2}$, then $w$ is a piece with respect to the original set of relators $R$. Thus $\abs{w} < 4\abs{r_2}/N$, since $G=\langle S \mid R \rangle$ is a $C'(4/N)$--group.

\textbf{Case 2:} $r_2\in C$. In this case we can write 
\[
r_2 = a^l\Pi_{i=1}^M v u_{i_v +i}
\]
for some relator $r'\in R_1$, cyclic subword $v\in W_{r'}$ and $l = f(\abs{r'})$. Since $u_i\neq u_j$ for all $i\neq j$ and $T$ is disjoint from $S$ and $\{a\}$, $p$ can intersect at most one copy of $w$. In particular, $\abs{p}\leq \abs{u_{j_1}} + \abs{u_{j_2}} + \abs{w} + k$. For large enough $V$, Lemma~\ref{lem:length-bound-ui} shows that $\abs{u_{j_1}}+\abs{u_{j_2}}\leq \abs{r_1}/M$. Since $f$ is a sublinear function, if $V$ is large enough, we have that $k < \abs{r_1}/M$. In particular, $\abs{p} < 2\abs{r_1}/M + \abs{w} < 4\abs{r_1}/M$. Analogously, $\abs{p} < 4 \abs{r_2}/M$.
\end{proof}

Our next goal is to show that $G'$ is a MLTG group.

\begin{lem}\label{lem:satisfies-mltg}
    For large enough $V$, the group $G' = \langle S' \mid R'\rangle$ satisfies the MLTG property.
\end{lem}
\begin{proof}
    Since $N, M, U\geq 36$, if  $V$ is large enough, then $G'$ is a $C'(1/9)$--group by Lemma~\ref{lem:small-cancellation-condition-satisfied}. Recall  that the original group $G$ does not satisfy the IPSC condition. Recalling that relators are cyclic words, the group $G'$ satisfies all conditions of Lemma~\ref{lem:combination-ipsc} by construction: \ref{prop:a}, \ref{prop:b} and \ref{prop:vi_is_subl} are straightforward. To see that \ref{prop:similar-size} holds, we observe the following. Since $f$ and $\log_2$ are sublinear functions, we have for large enough $V$ that
    \begin{align*}
        \abs{c_w} & = f(\abs{r}) + M\abs{w} + \sum_{i=1}^M \abs{u_{i_w+i}} \leq f(\abs{r}) + M\abs{r}/U + M\log_2(2M\abs{r}^3) + M\leq 2M\abs{r}/U. 
    \end{align*}
    Therefore, Lemma~\ref{lem:combination-ipsc} implies that $G'$ does not satisfy the IPSC condition. By Lemma~\ref{lem:equivalence:ipsc-sigma-compact-mltg} not satisfying the IPSC condition is equivalent to satisfying the MLTG property, which concludes the proof.
\end{proof}

\subsection{Proof of Main Theorem}\label{sec:main-theorem-proof}

We now show that the element $a\in G'$ is Morse but not loxodromic in any action of $G'$ on a hyperbolic 
space. To show that $a$ is Morse, we use Lemma~\ref{lem:4.14}. To show that $a$ cannot be loxodromic in any action of $G'$ on a hyperbolic space, we use that by Lemma~\ref{lem:shortsubpaths} any relator $r\in R_1$ has a cyclic subword $w\in W_r$ whose endpoints are close to each other. The existence of the relator $r_w$ then implies that the endpoints of $a^{f(\abs{r})}$ have to be too close.

\begin{lem}\label{lemma:a-is-morse}
    If $V$ is large enough, then the element $a\in \presu$ is Morse.
\end{lem}
\begin{proof}
    Lemma~\ref{lem:small-cancellation-condition-satisfied} shows that for $V$ large enough,  $G' = \presu$ is a $C'(1/9)$--group. 
    
    We consider the intersection function $\rho$ of the infinite path $\gamma$ labeled by $a^\infty$ in $\textrm{Cay}(G', S\cup S')$. 
    By construction, the length of the maximal subword of $\gamma$ and a relator $r'\in R'$ is $f(\abs{r'})$, while $\gamma$ has no common subword with any $r\in R$.  Since $f$ is sublinear, the intersection function $\rho$ must be as well.  By choosing $V$ large enough, we can ensure that $\rho(t) \leq t/3$.  By Lemma~\ref{lemma:geodesic_condition}, $\gamma$ is a geodesic, and by Lemma~\ref{lem:4.14}, $a$ is Morse.
\end{proof}

We next show that the element $a\in G'$  cannot  be loxodromic in any action of $G'$ on a hyperbolic space.
\begin{lem}\label{lem:a-not-lox}
    If $V$ is large enough, then the element $a\in G'$ is not loxodromic in any action of $G'$ on a hyperbolic space.
\end{lem}

\begin{proof}
    Let $G'\curvearrowright X$ be an isometric action on a $\delta$--hyperbolic space. We will show that $d_X(x_0, a^nx_0)$ grows sublinearly in $n$, which will imply that the element $a\in G'$ cannot be loxodromic. Let $g$ be a superlogarithmic function which grows slower than $f$, that is, $\lim_{x\to\infty} f(x)/g(x) = \infty$. For example, one could take $g(x) = \log\left(f(x)/\log(x)\right) \log(x)$.

    Let $N_0$ be the constant $N$ from Lemma~\ref{lem:shortsubpaths}, and $n \geq N_0$. Choose $r\in R_1$ such that $f(\abs{r})>n$ and $\abs{r}> N_0$. We view $r$ as an embedded cycle in $X$. By Lemma~\ref{lem:shortsubpaths}, there is a cyclic subword $w$ of $r$ satisfying $\abs{r}/L \leq \abs{w}\leq \abs{r}/U$ and $d_X(w^-,w^+)\leq g(\abs{r})$. In particular, $w\in W_r$.  Consider the relator $c_w\in C_r\subset C$, which we also view as an embedded cycle in $X$. This cycle has a subpath labeled by $a^{f(\abs{r})}$. However, by Lemma~\ref{lem:length-bound-ui} and the triangle inequality, then endpoints of $a^{f(\abs{r}})$ have distance at most     
    \begin{align*}
        Md_X(w^- w^+) + \sum_{i=0}^M d_{X}(u_{i_w + i}^-, u_{i_w+i}^+)\leq Mg(\abs{r}) + M\log_2(2M\abs{r}^3) + M =: A(\abs{r}).
    \end{align*}
    Since $f$ is superlogarithmic and grows faster than $g$, $A(x)$ grows slower than $f$.  Therefore,  $d_X(x_0, a^nx_0)$ grows sublinearly,  concluding the proof. 
\end{proof}

\begin{proof}[Proof of Theorem~\ref{thm:mainthm}]
    This follows immediately from Lemmas~\ref{lem:small-cancellation-condition-satisfied}, \ref{lem:satisfies-mltg}, \ref{lemma:a-is-morse}, and \ref{lem:a-not-lox}.
\end{proof}

\bibliography{mybib}

\newcommand{\etalchar}[1]{$^{#1}$}
\begin{thebibliography}{BFFGS22}

\bibitem[ACGH17]{ACHG:contraction_morse_divergence}
Goulnara~N Arzhantseva, Christopher~H Cashen, Dominik Gruber, and David Hume.
\newblock Characterizations of morse quasi-geodesics via superlinear divergence and sublinear contraction.
\newblock {\em Documenta Mathematica}, 22:1193--1224, 2017.

\bibitem[ACGH19]{arzhantseva2019negative}
Goulnara~N Arzhantseva, Christopher~H Cashen, Dominik Gruber, and David Hume.
\newblock Negative curvature in graphical small cancellation groups.
\newblock {\em Groups, Geometry, and Dynamics}, 13(2):579--632, 2019.

\bibitem[AD19]{ArzhDrutu:SC}
Goulnara Arzhantseva and Cornelia Dru\c{t}u.
\newblock Geometry of infinitely presented small cancellation groups and quasi-homomorphisms.
\newblock {\em Canad. J. Math.}, 71(5):997--1018, 2019.

\bibitem[AH20]{AbbottHume:gen_lox}
Carolyn~R. Abbott and David Hume.
\newblock The geometry of generalized loxodromic elements.
\newblock {\em Ann. Inst. Fourier (Grenoble)}, 70(4):1689--1713, 2020.

\bibitem[AH21]{AbbottHume:hyp_actions}
Carolyn~R. Abbott and David Hume.
\newblock Actions of small cancellation groups on hyperbolic spaces.
\newblock {\em Geom. Dedicata}, 212:325--363, 2021.

\bibitem[BF02]{BestFuji:WPD}
Mladen Bestvina and Koji Fujiwara.
\newblock Bounded cohomology of subgroups of mapping class groups.
\newblock {\em Geom. Topol.}, 6:69--89, 2002.

\bibitem[BFFGS22]{BalFFGenSisto}
Sahana Balasubramanya, Francesco Fournier-Facio, Anthony Genevois, and Alessandro Sisto.
\newblock Property {(NL)} for group actions on hyperbolic spaces.
\newblock {\em arXiv preprint arXiv:2212.14292}, 2022.

\bibitem[BH99]{BH}
Martin~R. Bridson and Andr\'{e} Haefliger.
\newblock {\em Metric spaces of non-positive curvature}, volume 319 of {\em Grundlehren der mathematischen Wissenschaften [Fundamental Principles of Mathematical Sciences]}.
\newblock Springer-Verlag, Berlin, 1999.

\bibitem[Bow12]{Bow:rel-hyp}
B.~H. Bowditch.
\newblock Relatively hyperbolic groups.
\newblock {\em Internat. J. Algebra Comput.}, 22(3):1250016, 66, 2012.

\bibitem[DG11]{DahGuirardel}
Fran\c~cois Dahmani and Vincent Guirardel.
\newblock The isomorphism problem for all hyperbolic groups.
\newblock {\em Geom. Funct. Anal.}, 21(2):223--300, 2011.

\bibitem[DGO17]{DGO:rotating_families}
Fran{\c{c}}ois Dahmani, Vincent Guirardel, and Denis Osin.
\newblock {\em Hyperbolically embedded subgroups and rotating families in groups acting on hyperbolic spaces}, volume 245.
\newblock American Mathematical Society, 2017.

\bibitem[ECH{\etalchar{+}}92]{ECHLPT}
David B.~A. Epstein, James~W. Cannon, Derek~F. Holt, Silvio V.~F. Levy, Michael~S. Paterson, and William~P. Thurston.
\newblock {\em Word processing in groups}.
\newblock Jones and Bartlett Publishers, Boston, MA, 1992.

\bibitem[Far98]{Farb:rel_hyp}
B.~Farb.
\newblock Relatively hyperbolic groups.
\newblock {\em Geom. Funct. Anal.}, 8(5):810--840, 1998.

\bibitem[Fin17]{fink2017morse}
Elisabeth Fink.
\newblock Morse geodesics in torsion groups.
\newblock {\em arXiv preprint arXiv:1710.11191}, 2017.

\bibitem[GH21]{GH:acylindrical}
Anthony Genevois and Camille Horbez.
\newblock Acylindrical hyperbolicity of automorphism groups of infinitely ended groups.
\newblock {\em Journal of Topology}, 14(3):963--991, 2021.

\bibitem[Gro87]{Gromov-hyp}
Misha Gromov.
\newblock Hyperbolic groups.
\newblock In {\em Essays in group theory}, volume~8 of {\em Math. Sci. Res. Inst. Publ.}, pages 75--263. Springer, New York, 1987.

\bibitem[Gru15]{Gruber:SQ}
Dominik Gruber.
\newblock Infinitely presented {$C(6)$}-groups are {SQ}-universal.
\newblock {\em J. Lond. Math. Soc. (2)}, 92(1):178--201, 2015.

\bibitem[HSZ24]{HeSprianoZbinden:sigma}
Vivian He, Davide Spriano, and Stefanie Zbinden.
\newblock Sigma-compactness of morse boundaries in morse local-to-global groups and applications to stationary measures.
\newblock {\em arXiv preprint arXiv:2407.18863}, 2024.

\bibitem[LSLS77]{lyndon1977combinatorial}
Roger~C Lyndon, Paul~E Schupp, RC~Lyndon, and PE~Schupp.
\newblock {\em Combinatorial group theory}, volume 188.
\newblock Springer, 1977.

\bibitem[OOS09]{OlshankiiOsinSapir:lacunary}
Alexander~Yu. Olshanskii, Denis~V. Osin, and Mark~V. Sapir.
\newblock Lacunary hyperbolic groups.
\newblock {\em Geom. Topol.}, 13(4):2051--2140, 2009.
\newblock With an appendix by Michael Kapovich and Bruce Kleiner.

\bibitem[Osi06]{Osin:rel_hyp}
Denis~V. Osin.
\newblock Relatively hyperbolic groups: intrinsic geometry, algebraic properties, and algorithmic problems.
\newblock {\em Mem. Amer. Math. Soc.}, 179(843):vi+100, 2006.

\bibitem[Osi16]{O:acylindrical}
Denis Osin.
\newblock Acylindrically hyperbolic groups.
\newblock {\em Transactions of the American Mathematical Society}, 368(2):851--888, 2016.

\bibitem[Pau91]{Paulin}
Fr\'ed\'eric Paulin.
\newblock Outer automorphisms of hyperbolic groups and small actions on {${\bf R}$}-trees.
\newblock In {\em Arboreal group theory ({B}erkeley, {CA}, 1988)}, volume~19 of {\em Math. Sci. Res. Inst. Publ.}, pages 331--343. Springer, New York, 1991.

\bibitem[RST22]{russellsprianotran:thelocal}
Jacob Russell, Davide Spriano, and Hung~Cong Tran.
\newblock The local-to-global property for {M}orse quasi-geodesics.
\newblock {\em Math. Z.}, 300(2):1557--1602, 2022.

\bibitem[Sel95]{Sela}
Z.~Sela.
\newblock The isomorphism problem for hyperbolic groups. {I}.
\newblock {\em Ann. of Math. (2)}, 141(2):217--283, 1995.

\bibitem[Sis16]{S:hypemb}
Alessandro Sisto.
\newblock Quasi-convexity of hyperbolically embedded subgroups.
\newblock {\em Math. Z.}, 283(3-4):649--658, 2016.

\bibitem[Zbi23]{Z:small_cancellation}
Stefanie Zbinden.
\newblock Small cancellation groups with and without sigma-compact morse boundary.
\newblock {\em arXiv preprint arXiv:2307.13325}, 2023.

\end{thebibliography}
\bibliographystyle{alpha}
\end{document}